\documentclass{amsart}
\usepackage{amssymb,color,epsf}
\usepackage [cmtip,arrow]{xy}
\xyoption{all}
\usepackage {pb-diagram,pb-xy}
\usepackage{enumerate}

\ifx\pdfpageheight\undefined
   \usepackage[dvips,colorlinks=true,linkcolor=blue,citecolor=red,%
      urlcolor=green]{hyperref}
   \usepackage[dvips]{graphicx}
   \makeatletter
   \edef\Gin@extensions{\Gin@extensions,.mps}
   \makeatother
\else
 \usepackage[pdftex]{graphicx}
   \usepackage[bookmarksopen=false,pdftex=true,breaklinks=true,%
      backref=page,pagebackref=true,plainpages=false,%
      hyperindex=true,pdfstartview=FitH,colorlinks=true,%
      pdfpagelabels=true,colorlinks=true,linkcolor=blue,%
      citecolor=red,urlcolor=green,hypertexnames=false%
      ]%
   {hyperref}
\fi
\usepackage{bm}

\newcommand{\Z}{\mathbb{Z}}

\usepackage{amsmath,amssymb,amsfonts}
\newtheorem{theorem}{Theorem}
\newtheorem{lemma}[theorem]{Lemma}
\newtheorem{corollary}[theorem]{Corollary}

\theoremstyle{definition}
\newtheorem{definition}[theorem]{Definition}

\newtheorem{notation}[theorem]{Notation}

\theoremstyle{remark}
\newtheorem{remark}[theorem]{Remark}

\definecolor{DarkBlue}{rgb}{0,0.1,0.55}

\numberwithin{equation}{section}

\newcommand {\hide}[1]{}

\newcommand{\nocomma}{}
\newcommand{\noplus}{}
\newcommand{\tmop}[1]{\ensuremath{\operatorname{#1}}}
\newcommand{\bidegree}{\mathrm{bidegree}}
\begin{document}

\title[Spectral Sequences, Exact Couples and Persistent Homology]
{Spectral Sequences, Exact Couples and Persistent Homology of
Filtrations}

\author{Saugata Basu}
\address{Department of Mathematics,
Purdue University, West Lafayette, IN 47906, U.S.A.}
\email{sbasu@math.purdue.edu}
\urladdr{www.math.purdue.edu/~sbasu}
\author{Laxmi Parida}
\address{
Computational Genomics,\\
IBM T.J. Watson Research Center,\\
Yorktown Heights, NY 10598.
}
\email{parida@us.ibm.com}
\urladdr{researcher.ibm.com/person/us-parida}

\thanks{The first author was partially supported by NSF grants
CCF-1319080 and DMS-1161629.}

\maketitle

\begin{abstract}
  In this paper we study the relationship between a very classical algebraic
  object associated to a filtration of topological spaces, namely a spectral
  sequence introduced by Leray in the 1940's, and a more recently invented
  object that has found many applications -- namely, its persistent homology
  groups. We show the existence of a long exact sequence of groups linking
  these two objects and using it derive formulas expressing the dimensions of
  each individual groups of one object in terms of the dimensions of the
  groups in the other object. The main tool used to mediate between these
  objects is the notion of exact couples first introduced by Massey in 1952.
\end{abstract}

\section{Introduction}

Given a topological space $X$ (which for the purposes of the current paper
will be taken to be a finite CW-complex) a finite filtration, $\mathcal{F}$ of
$X$, is a sequence of subspaces
\begin{eqnarray*}
  \emptyset = X_{- 1} = X_0 \subset X_1 \subset \cdots \subset X_N = X_{N + 1}
  = \cdots = X &  & 
\end{eqnarray*}
(we will denote the subspace $X_p$ in the sequence by $\mathcal{F}_- X$). A
very classical technique in algebraic topology for computing topological
invariants of a space $X$ is to consider a filtration $\mathcal{F}$ of $X$
where the successive spaces $\mathcal{F}_s X$ capture progressively more and
more of the topology of $X$. For example, in case $X$ is a CW-complex one can
take for $\mathcal{F}_p X$ the $p$-th skeleton
$\ensuremath{\operatorname{sk}}_p (X)$ consisting of all cells of dimension at
most $p$. More generally, given a cellular map $f : X \rightarrow Y$, one can
take for $\mathcal{F}_p X$ the inverse image under $f$ of
$\mathrm{sk}_p (Y)$. One then associates to this sequence a
sequence of algebraic objects which in nice situations is expected to
``converge'' (in an appropriate sense) to the topological invariant (such as
the homology or cohomology groups) associated to $X$ itself, directly
computing which is often an intractable problem. This sequence of algebraic
approximations is called a \emph{spectral sequence} associated to the
filtration $\mathcal{F}$, and was first introduced by Leray {\cite{Leray46}}
in 1946 (see also the book by Dieudonn{\'e} {\cite[page 137]{Dieudonne-book}}
for a comprehensive historical survey).

Spectral sequences are now ubiquitous in mathematics. A typical application
which is common in discrete geometry, as well as in quantitative real
algebraic geometry, is to use the initial terms of a certain spectral sequence
to give upper bounds on the topological complexity (for example, the sum of
Betti numbers) of the object of interest $X$ (often a semi-algebraic subset of
some $\mathbb{R}^n$) (see for example, \cite{BZ,GVZ04} for
applications of this kind). Spectral sequences also have algorithmic
applications in the context of computational geometry (see for example
\cite{Basu5}).

Much more recently the notion of \emph{persistent homology}
\cite{Edelsbrunner-Harer2010,Weinberger2011} associated to a
filtration has become an important tool in various applications. In contrast
to spectral sequences discussed in the previous paragraph, the emphasis here
is not so much on studying the topology of the final object $X$, but rather on
the intermediate spaces of the filtration. Indeed the final object $X$ in many
cases is either contractible or homologically trivial. For example, this is
the case for filtrations arising from alpha-complexes introduced by 
Edelsbrunner et al in \cite{Edelsbrunner-et-al-1983}. The \emph{persistent
homology groups}
  (see Definition \ref{def:persistent} below for a precise definition) are defined such that
their dimensions equal the dimensions of spaces of homological cycles that
appear at a certain fixed point of the filtration $\mathcal{F}$ and disappear
at a certain (later) point. The homological cycles that persist for long
intervals often carry important information about the underlying data sets
that give rise to the filtration, and this is why computing them is important
in practice. We refer the reader to survey articles
\cite{Carlsson2008,Ghrist2008,Weinberger2011} for details
regarding these applications.

While spectral sequences and persistent homology were invented for entirely
different purposes as explained above, they are both associated to filtrations
of topological spaces and it is natural to wonder about the exact relationship
between these two notions, and in particular, whether the dimensions of the
groups appearing in the spectral sequence of a filtration carries any more
information than the dimensions of the persistent homology groups of the same
filtration. One of the results in this paper (Theorem
\ref{thm:spectral-to-persistent} below) shows that this is not the case, and
the dimensions of the groups appearing in the spectral sequence of a
filtration can be recovered from the dimensions of its persistent homology
groups.  It has been observed by several authors (see for example,
\cite{Carlsson2008,Edelsbrunner-Harer2010,Ghrist2008})
that there exists a close connection between the spectral sequence of a
filtration and its persistent homology groups. The goal of this paper is to
make precise this relationship -- in particular, to derive formulas which
expresses the dimensions of each group appearing in the spectral sequence of a
filtration in terms of the persistent Betti numbers and vice versa.

For the rest of the paper we fix a field $k$ of coefficients and all homology
groups will be taken with coefficients in $k$, and we will omit the dependence
on $k$ from our notation henceforth. Given a \emph{finite} filtration
$\mathcal{F}$, given by $\emptyset = \cdots = X_{- 1} = X_0 \subset X_1
\subset \cdots \subset X_s \subset X_{s + 1} \subset \cdots \subset X_N = X_{N
+ 1} = \cdots = X$ of finite CW-complexes, there is an associated homology
spectral sequence 
\[
\left(E^{(r)} (\mathcal{F}) = \bigoplus_{p,q} E^{(r)}_{p,q} (\mathcal{F}), d^{(r)}\right)_{r \geq 1}
\] 
(see Definition \ref{def:spectral-sequence} in \S \ref{sec:spectral}), where each
$E^{(r)}_{p,q} (\mathcal{F})$ is a finite dimensional $k$-vector
space. 

For each fixed $r \geq 1 $, the set of groups $(E^{(r)}_{p,q}
(\mathcal{F}))_{p,q}$ is sometimes called the \emph{$r$-th page} of the
spectral sequence of $\mathcal{F}$, and they converge in an appropriate sense
explained later to the homology of $X$ as $r \rightarrow \infty$. Similarly,
the persistent homology groups, $(H_n^{s, t} (\mathcal{F}))_{n, s, t}$
(see Definition \ref{def:persistent} in \S \ref{sec:persistent}) are also finite
dimensional $k$-vector spaces, and correspond to the images under the linear
maps induced by the inclusions$(X_s\hookrightarrow X_t)_{s,t\in\mathbb{Z}, s\leq t}$.

It is natural to ask whether the sequence of dimensions $(\dim_k E^{(r)}_{p,q} (\mathcal{F}))_{p,q \in \mathbb{Z}, r \geq 1}$ determine the sequence
$(b_{n}^{s, t} (\mathcal{F}) = \dim_k H_n^{s, t} (\mathcal{F}))_{n, s, t
\in \mathbb{Z}, s \leq t}$ and conversely. In the book
{\cite{Edelsbrunner-Harer2010}} the authors give a relation between these sets
of numbers (see the {{\em Spectral Sequence Theorem}}, page 171 in
{\cite{Edelsbrunner-Harer2010}}), but this theorem does not produce a simple
formula expressing the dimension of each group,  $E^{(r)}_{p,q}
(\mathcal{F})$, in the spectral sequence,  in terms of the persistent Betti
numbers $b_{n}^{s, t} (\mathcal{F})$. The statement of this theorem in
{\cite{Edelsbrunner-Harer2010}} has an error and there are some terms missing
from the right-hand side of the equality. We fix this error and obtain the
corrected version as a corollary to one of our theorems (see Corollary
\ref{cor:main} and Remark \ref{rem:correction} below).

We study the relationship between the persistent homology groups of a
filtration and the homology spectral sequence via another classical tool in
algebraic topology -- namely \emph{exact couples}, first introduced by
Massey \cite{Massey52}. This gives us a simple way to relate the persistent
Betti numbers with the dimensions of the groups occurring in the spectral
sequence associated to the filtration. A hidden  motivation behind this paper
is to clarify the relationship between the spectral sequence groups and the
persistent homology groups without getting bogged down in a morass of indices
and a large number of intermediate groups of cycles and boundaries. The exact
couple formulation is very elegant (and economical) in this respect, and the
only extra groups (i.e. groups other than the  various $E^{(r)}_{p,q}(\mathcal{F})$) that appear are subgroups of the homology groups of spaces
appearing in the filtration. In fact, we will identify these extra groups with
the persistent homology groups of the filtration (see Lemma \ref{lem:basic}
below). We refer the reader also to the fundamental paper of Boardman \cite{Boardman}
where the technique of unraveling exact couples is explained at length.

Our main results are the following.

\begin{theorem}
  \label{thm:spectral-to-persistent} 
  Let $\mathcal{F}$ denote the 
  filtration,  $\emptyset = \cdots = X_{- 1} = X_0 \subset X_1 \subset \cdots
  \subset X_p \subset X_{p + 1} \subset \cdots \subset X_N = X_{N + 1} =
  \cdots = X$ where each $X_i$ is a finite CW-complex.  Then,  for 
  %%each $r >0$, $n, s \geq 0$,
  $r,p,q, \in \Z, r\geq 1$ and setting $n=p+q$,
  \begin{eqnarray*}
    \dim_k E^{(r)}_{p,q} (\mathcal{F}) & = & (b_n^{p, p + r - 1}
    (\mathcal{F}) - b_n^{p - 1, p + r - 1} (\mathcal{F})) + (b_{n - 1}^{p - r,
    p - 1} (\mathcal{F}) - b_{n - 1}^{p - r, p} (\mathcal{F})) . 
    %%\label{eqn:main}
  \end{eqnarray*}
\end{theorem}

Theorem \ref{thm:spectral-to-persistent} is a direct consequence of the
existence of a long exact sequence linking the the spectral groups
$E^{(r)}_{p,q} (\mathcal{F})$ and the persistent homology groups. This long
exact sequence which appears in Theorem \ref{thm:basic} below encapsulates the
exact relationship between these groups.

We also define (following {\cite{Edelsbrunner-Harer2010}}) for $i < j$
the \emph{persistent multiplicities} of the filtration $\mathcal{F}$ by
\begin{eqnarray}
  \mu^{i, j}_n (\mathcal{F}) & = & (b^{i, j - 1}_n (\mathcal{F}) - b^{i, j}_n
  (\mathcal{F})) - (b^{i - 1, j - 1}_n (\mathcal{F}) - b^{i - 1, j}_n
  (\mathcal{F})) .  \label{eqn:mu}
\end{eqnarray}
In the language of persistence theory, the number $\mu^{i, j}_n (\mathcal{F})$
has the following interpretation. It is the dimension of the $k$-vector space
spanned by the $n$-dimensional homology cycles, which are born at time $i$ and
gets killed at time $j$ (see {\cite{Edelsbrunner-Harer2010}}).

We obtain the following corollary to Theorem \ref{thm:spectral-to-persistent}
recovering (and correcting slightly) a result in
\cite[page 171, \emph{Spectral Sequence
Theorem}]{Edelsbrunner-Harer2010}.
%% using a slightly different notation (see also Remark\ref{rem:indices} below).

\begin{corollary}
  \label{cor:main} The following relation holds for every $r \geq 1$, and all $n
  \geq 0$.
  \begin{eqnarray*}
    \sum_{p+q=n} \dim_k E^{(r)}_{p, q} (\mathcal{F}) & = & \sum_{j - i \geq r}
    (\mu_n^{i, j} (\mathcal{F}) + \mu_{n - 1}^{i, j} (\mathcal{F})) + b_n (X).
  \end{eqnarray*}
\end{corollary}

\begin{remark}
  \label{rem:correction}In {\cite{Edelsbrunner-Harer2010}} the statement of
  the ``Spectral Sequence Theorem'' is stated incorrectly as
  \begin{eqnarray*}
    \sum_{p+q=n} \dim_k E^{(r)}_{p, q} (\mathcal{F}) & = & \sum_{j - i \geq r}
    \mu_n^{i, j} (\mathcal{F}),
  \end{eqnarray*}
  and in a previous version of this paper we had a wrong proof of the above
  erroneous statement. We thank Ana Romero for pointing out a counter-example
  to this statement which appears in \cite{Romero-et-al-2013} and which
  directed us to the correct statement. A different corrected version of the
  statement also appears in \cite{Romero-et-al-2013}, where the left hand
  side is the sum of the dimensions of a different set of groups. To our
  knowledge the equality in Corollary \ref{cor:main} is new.
\end{remark}

We next show how to express the persistent Betti numbers in terms of
dimensions of the vector spaces appearing in spectral sequences associated to
truncations of the filtrations $\mathcal{F}$.

\begin{notation}
  Given a finite filtration $\mathcal{F}$, given by \ $\emptyset = \cdots =
  X_{- 1} = X_0 \subset X_1 \subset \cdots \subset X_s \subset X_{s + 1}
  \subset \cdots \subset X_N = X_{N + 1} = \cdots$ of finite CW-complexes, let
  for $0 \leq t \leq N$, $\mathcal{F}_{\leq t}$ denote the truncated
  filtration
  \[ \emptyset = \cdots = X_{- 1} = X_0 \subset X_1 \subset \cdots \subset
     X_t = X_{t + 1} = \cdots \]
\end{notation}

We prove

\begin{theorem}
  \label{thm:persistence-in-terms-of-spectral}For each, $n, s, t \geq 0$, $s
  \leq t $,
  \begin{eqnarray*}
    b_n^{s, t} (\mathcal{F}) & = & \sum_{0 \leq i \leq s} \dim_k (E_{i,
    n-i}^{(\max (i, t - i + 1))} (\mathcal{F}_{\leq t})) .
  \end{eqnarray*}
\end{theorem}

\begin{remark}
\label{rem:CW}
Even though we state our results for filtrations of CW-complexes for clarity, the results and
their proofs are valid for filtrations of abstract finite dimensional complexes of 
$k$-vector spaces as well.
\end{remark}

The rest of the paper is organized  as follows. In \S
\ref{sec:persistent}, we recall the definitions of  persistent homology
groups of a filtration and the persistent Betti numbers. In \S
\ref{sec:exact-couples}, we recall the definitions of exact couples and prove
their basic properties. In \S \ref{sec:spectral}, we recall how to
associate an exact couple to a filtration via the long exact homology sequence
of a pair. We also establish the relationships between the groups in the
derived couples and the persistent homology groups of the same filtration.
This allows us to prove (rather easily) the main theorems in \S
\ref{sec:proofs}.

We do not assume any prior knowledge on spectral sequences and only assume
basic linear and only a slight familiarity with homological algebra as
prerequisites.

\section{Persistent Homology}\label{sec:persistent}

We now recall the precise definition of the persistent homology
groups associated to a filtration
\cite{Edelsbrunner-Harer2010,Weinberger2011}.

Let $\mathcal{F}$ denote the filtration of spaces $\emptyset = \cdots = X_{-
1} = X_0 \subset X_1 \subset \cdots \subset X_s \subset X_{s + 1} \subset
\cdots \subset X_N = X_{N + 1} = \cdots = X$ of spaces as in the last section.

\begin{notation}
\label{not:inclusion}
  For $
  %%0 \leq 
  s \leq t 
  %%\leq N
  $, we let $i_n^{s, t} : H_n (X_s) \longrightarrow
  H_n (X_t)$, denote the \ homomorphism induced by the inclusion $X_s
  \hookrightarrow X_t$.
\end{notation}

With the same notation as in the previous section we define:

\begin{definition}
\label{def:persistent}
  {\cite{Edelsbrunner-Harer2010}} For each triple $(n, s, t)$ with 
  $
  %%0 \leq 
  s \leq t 
  %%\leq N
  $ the corresponding  \emph{persistent homology
  group}, $H_n^{s, t} (\mathcal{F})$ is defined by
  \begin{eqnarray*}
    H_n^{s, t} (\mathcal{F}) & = & \ensuremath{\operatorname{Im}} (i_n^{s, t}).
  \end{eqnarray*}
  Note that $H_n^{s, t} (\mathcal{F}) \subset H_n (X_t)$, and $H_n^{s, s}
  (\mathcal{F}) = H_n (X_s)$.
\end{definition}

\begin{notation}
  We denote by ${b_n^{s, t} (\mathcal{F}) = \dim_k H_n^{s, t}
  (\mathcal{F})}$.
\end{notation}

\section{Algebra of exact Sequences and couples}\label{sec:exact-couples}

There are several ways of defining the homology spectral sequence associated
to a filtration $\mathcal{F}$. As mentioned in the introduction we prefer the
approach due to Massey \cite{Massey52} (see also {\cite{Mcleary}}) using
\emph{exact couples} (defined below) since it avoids defining various
intermediate groups of cycles and boundaries and clarifies at a top level the
close relationship between the groups appearing in the spectral sequence and
the persistent homology groups. The exact relationships between the dimensions
of the groups in the spectral sequence and the persistent Betti numbers can
then be read off with minimal extra effort.

We begin by recalling a few basic notions.

Recall that a sequence of linear maps between $k$-vector spaces
\[ \cdots \longrightarrow V_{i + 1} \mathop{\longrightarrow}\limits^{f_{i +
   1}} V_i \mathop{\longrightarrow}\limits^{f_i} V_{i - 1}
   \mathop{\longrightarrow}\limits^{f_{i - 1}} V_{i - 2}
   \mathop{\longrightarrow}\limits^{f_{i - 2}} \cdots \]
is said to be \emph{exact} if for each $i$, $\ker f_i
=\ensuremath{\operatorname{Im}}f_{i + 1}$.

The following lemma is an easy consequence of the exactness property.

\begin{lemma}
  \label{lem:exact} Given an exact sequence
  \[ V_2 \mathop{\longrightarrow}\limits^{f_2} V_1
     \mathop{\longrightarrow}\limits^{f_1} V_0
     \mathop{\longrightarrow}\limits^{f_0} V_{- 1}
     \mathop{\longrightarrow}\limits^{f_{- 1}} V_{- 2} \]
  where each $V_i$ is a finite dimensional $k$-vector space
  \begin{eqnarray*}
    \dim_k (V_0) & = & (\dim_k (V_1) - \dim_k  (\ensuremath{\operatorname{Im}}
    (f_2)))   + (\dim_k (V_{- 1}) - \dim_k (\ensuremath{\operatorname{Im}}
    (f_{- 1}))) .
  \end{eqnarray*}
\end{lemma}

\begin{proof}
The lemma follows from the following sequence of
inequalities of which the first one is from basic linear algebra, and the
remaining consequences of the exactness of the given sequence.
\begin{eqnarray*}
     \dim_k (V_0) & = & \dim_k (\ker (f_0)) + \dim_k (\tmop{Im} (f_0))\\
     & = & \dim_k (\tmop{Im} (f_1)) + \dim_k (\ker (f_{- 1}))\\
     & = & (\dim_k (V_1) - \dim_k  (\ker (f_1))) \noplus \noplus + (\dim_k
     (V_{- 1}) - \dim_k (\tmop{Im} (f_{- 1})))\\
     & = & (\dim_k (V_1) - \dim_k  (\tmop{Im} (f_2))) \noplus \noplus +
     (\dim_k (V_{- 1}) - \dim_k (\tmop{Im} (f_{- 1}))).
   \end{eqnarray*}
\end{proof}

%%\hspace*{\fill}$\Box$\medskip

We now define following Massey {\cite{Massey52}} (see also {\cite[page 37]{Mcleary}}):

\begin{definition}
\label{def:exact-couple}
  An \emph{exact couple}, $\mathcal{C}$, consists of two finite
  dimensional $k$-vector spaces $E, D$ and linear maps $\partial, i$ and $j$
  such that the following triangular diagram is exact at each vertex.
  
\[
\xymatrix{
&E \ar[rd]^\partial& \\
D\ar[ur]^j && D\ar[ll]^i 
}
\]
 
\end{definition}

Given an exact couple
\begin{eqnarray*}
\mathcal{C}  &=&
 \xymatrix{
&E \ar[rd]^\partial& \\
D\ar[ur]^j&& D\ar[ll]^i 
}
\end{eqnarray*}

the exactness at each vertex implies that the map $d = j \circ \partial$ is a
\emph{differential} i.e. $d^2 = j \circ \partial \circ j \circ \partial = j
\circ (\partial \circ j) \circ \partial = 0$.

Let $E' = H (E, d) = \ker d /\ensuremath{\operatorname{Im}}d$, and $D'
=\ensuremath{\operatorname{Im}}i$.

Then, there exists {{\em induced linear maps\/}} $\partial' : E'
\longrightarrow D'$, $i' : D' \longrightarrow D'$, $j' : D' \longrightarrow
E'$ \ defined as follows.

First notice that for any element $x + d E \in E'$, where $x \in \ker (d)$, \
we have by the exactness at $E$ of the couple \ $\mathcal{C}$, that $\partial
(x) \in \ensuremath{\operatorname{Im}} (i)$ (since $x \in \ker (d)$ is
equivalent to 
$j \circ \partial (x) = 0$,
%%sb changes
 which implies that $\partial (x) \in
\ker (j) =\ensuremath{\operatorname{Im}} (i)$).

Now define
\begin{eqnarray}
\nonumber
  \partial' (x + d E) & = & \partial (x) \in \ensuremath{\operatorname{Im}}
  (i) = D', \ensuremath{\operatorname{for}}\ensuremath{\operatorname{all}}x
  \in \ker (d), \\
  \label{eqn:defnofderivedcouple}
  i' & = & i \mid_{\ensuremath{\operatorname{Im}}i}, \\
  \nonumber
  j' (i (x)) & = & j (x) + d E,
  \ensuremath{\operatorname{for}}\ensuremath{\operatorname{all}}i (x) \in
  \ensuremath{\operatorname{Im}}i. 
\end{eqnarray}
It follows easily from these definitions that the $k$-vector spaces $E', D'$,
and the linear maps $\partial', i', j'$ also form an exact couple. In other
words the following diagram is exact.
\[
\xymatrix{
&E' \ar[rd]^{\partial'}& \\
D'\ar[ur]^{j'}&& D'\ar[ll]^{i'} 
}
\]

\begin{definition}
  \label{def:derived} The exact couple
  \begin{eqnarray*}
    \mathcal{C}' & = & 
    \xymatrix{
&E' \ar[rd]^{\partial'}& \\
D'\ar[ur]^{j'}&& D'\ar[ll]^{i'} 
}
  \end{eqnarray*}
  is called the (first) \emph{derived couple} of the
  exact couple $\mathcal{C}$.
\end{definition}

\begin{notation}
  Given an exact couple $\mathcal{C}=\mathcal{C}^{(1)}$, we denote for
  each $r \geq 1$, $\mathcal{C}^{(r + 1)} = (\mathcal{C}^{(r)})' .$
\end{notation}

\section{Homology Spectral sequence associated to a
filtration}\label{sec:spectral}

We now use the notion of an exact couple introduced in the last section to
define the homology spectral sequence of a filtration $\mathcal{F}$ given by
$\emptyset = \cdots = X_{- 1} = X_0 \subset X_1 \subset \cdots \subset X_p
\subset X_{p + 1} \subset \cdots \subset X_N = X_{N + 1} = \cdots = X$ \ where
each $X_i$ is a finite CW-complex.

\begin{notation}
  For 
  %%$0 \leq s \leq N$
  $p \leq N$,
  denoting $F_p H_n (X) =\ensuremath{\operatorname{Im}}
  (i_n^{p, N})$ we have a filtration of the vector space $H_n (X) = H_n (X,
  k)$ given by
  
  \[ 0 = F_0 H_n (X) \subset F_1 H_n (X) \subset \cdots F_p H_n (X) \subset
     F_{p + 1} H_n (X) \subset \cdots \subset F_N H_n (X) = H_n (X) \]
\end{notation}

Defining
\begin{eqnarray}
  \ensuremath{\operatorname{Gr}}_{p} (H_n (X)) & = & F_p H_n (X) / F_{p -
  1} H_n (X), \nonumber\\
  & = & \ensuremath{\operatorname{Im}} (i_n^{p, N})
  /\ensuremath{\operatorname{Im}} (i_n^{p - 1, N})  \label{eqn:graded}
\end{eqnarray}
we have
\begin{eqnarray*}
  \dim_k H_n (X) & = & \sum_{p \geq 0} \dim_k
  \ensuremath{\operatorname{Gr}}_{p} (H_n (X)) .
\end{eqnarray*}
Recall also the homology exact sequence of the pair $(X_p, X_{p - 1})$. This is the long exact sequence
in homology induced from the short exact sequence
\[
0\rightarrow C_\bullet(X_{p-1}) \xrightarrow{i_{p-1,\bullet}} C_\bullet(X_p) \xrightarrow{j_{p,\bullet}} C_\bullet(X_p,X_{p-1}) \rightarrow 0
\]
of chain complexes (see for example \cite[\S 2, No. 3, Th\'eor\`eme 1]{Bourbaki-Chapter10}). 

The induced long exact sequence in homology is the following  (using the same notation for the induced 
homomorphisms as in 
\cite[\S 2, No. 3, Th\'eor\`eme 1]{Bourbaki-Chapter10}):

$$\displaylines{
 \cdots \rightarrow H_n (X_{p- 1}) \xrightarrow{H_n(i_{p-1,\bullet})} H_n 
   (X_p) \xrightarrow{H_n(j_{p,\bullet})} H_n (X_p, X_{p - 1}) \cr
   \xrightarrow{H_n(i_{p,\bullet},j_{p,\bullet})} H_{n - 1} (X_{p - 1})
   \rightarrow \cdots 
   }
$$

We now denote
\begin{eqnarray}
  E^{(1)}_{p,n-p} (\mathcal{F}) & = & H_n (X_p, X_{p - 1}),  \label{eqn:E1}\\
  D^{(1)}_{p, n-p} (\mathcal{F}) & = & H_n (X_p),  \label{eqn:D1} \\
  \nonumber 
  i_{p,n-p} &=& H_n(i_{p,\bullet}), \\ \nonumber
  j_{p,n-p} &=& H_n(j_{p,\bullet}), \\ \nonumber
  \partial_{p,n-p} &=& H_n(i_{p,\bullet},j_{p,\bullet}). 
\end{eqnarray}

We now bundle together the various $E^{(1)}_{p, n-p} (\mathcal{F})$ (respectively,
$D^{(1)}_{p, n-p} (\mathcal{F})$) into one bigraded vector space $E^{(1)}
(\mathcal{F})$ (respectively, $D^{(1)} (\mathcal{F})$) by defining
\begin{eqnarray*}
  E^{(1)} (\mathcal{F}) & = & \bigoplus_{p,q} E^{(1)}_{p,q} (\mathcal{F}),\\
  D^{(1)} (\mathcal{F}) & = & \bigoplus_{p,q} D^{(1)}_{p,q} (\mathcal{F}) .
\end{eqnarray*}
Note that the $k$-vector spaces $E^{(1)} (\mathcal{F})$ and $D^{(1)} (\mathcal{F})$ \
are \emph{bigraded}  and each of them is a direct sum of homogeneous subspaces
indexed by the pairs $(p,	q)$.  We refer the reader to
{\cite[Chapter 2, \S 11.2)]{Bourbaki-algebra}}  for background on
graded vector spaces and graded homomorphisms (linear maps) between them.

In particular, a \emph{bigraded homomorphism}  $\phi : \bigoplus_{p,q} V_{p,q}
\longrightarrow \bigoplus_{p,q} W_{p,q}$ between two bigraded vector spaces
is said to have \emph{bidegree} $(d,d')$ if $\phi (V_{p,q}) \subset W_{p + d, q +
d'}$.

We denote by $\partial$ (respectively, $i, j$) the bigraded linear maps
$\bigoplus_{p,q} \partial_{p, q}$ (respectively, $\bigoplus_{p, q} i_{p,q},
\bigoplus_{p,q} j_{p,q}$), 
%%sb changes
and get an exact couple
\begin{eqnarray*}
  \mathcal{C} (\mathcal{F}) & = & 
  \xymatrix{
&E^{(1)}(\mathcal{F}) \ar[rd]^{\partial}& \\
D^{(1)}(\mathcal{F})\ar[ur]^{j}&& D^{(1)}(\mathcal{F})\ar[ll]^{i} 
}.
\end{eqnarray*}
Note that the bidegrees of the various linear maps $\partial$, $i,
p$, can be read off from the subscripts of the source and targets of the following homomorphisms.
\[ \partial_{p,q} : E^{(1)}_{p,q} (\mathcal{F}) \longrightarrow D^{(1)}_{p - 1, q} (\mathcal{F}), \]
\[ i_{p,q} : D^{(1)}_{p,q} (\mathcal{F}) \longrightarrow D^{(1)}_{p+1, q-1}
   (\mathcal{F}), \]
\[ j_{p,q}: D^{(1)}_{p,q} (\mathcal{F}) \longrightarrow E^{(1)}_{p,q} (\mathcal{F}) .
\]

It is evident from the above that:
\begin{eqnarray}
\nonumber
  \mathrm{bidegree} (\partial) & = & (- 1, 0),\\
 \label{eqn:bidegree-1} 
  \mathrm{bidegree} (i) & = & (1, -1),\\
\nonumber
  \mathrm{bidegree} (j) & = & (0, 0) .
\end{eqnarray}
On deriving the exact couple $\mathcal{C} (\mathcal{F})$, $r$ times,  we obtain
for each $r \geq 1$, the couple
\begin{eqnarray*}
  \mathcal{C}^{(r)} (\mathcal{F}) & = & 
  \xymatrix{
&E^{(r)}(\mathcal{F}) \ar[rd]^{\partial^{(r)}}& \\
D^{(r)}(\mathcal{F})\ar[ur]^{j^{(r)}}&& D^{(r)}(\mathcal{F})\ar[ll]^{i^{(r)}} 
}
    .
\end{eqnarray*}
\begin{definition}
\label{def:spectral-sequence}
  The sequence of pairs $(E^{(r)} (\mathcal{F}), d^{(r)}
  (\mathcal{F}))_{r \geq 1}$ is called the (homology) \emph{spectral
  sequence associated to the filtration $\mathcal{F}$}.
\end{definition}

Notice that  
\[
E^{(r + 1)} (\mathcal{F}) = (E^{(r)} (\mathcal{F}))' = H (E^{(r)}
(\mathcal{F}), d^{(r)}),
\] 
\[
D^{(r + 1)} (\mathcal{F}) = (D^{(r)}
(\mathcal{F}))' = i^{(r)} (D^{(r)} (\mathcal{F})),
\]
and both $E^{(r+1)}$ and $D^{(r+1)}$ inherit a bigrading from $E^{(r)}$ and $D^{(r)}$.
We index the homogeneous components of  these bigradings 
such that for each pair $(p, q)$:
\begin{enumerate}
  \item $E^{(r + 1)}_{p,q} (\mathcal{F})$ is a subquotient (i.e. quotient
  of a subspace) of $E^{(r)}_{p,q} (\mathcal{F})$, and
   \item 
   $D^{(r + 1)}_{p,q} (\mathcal{F})$ is a subspace of $D^{(r)}_{p,q} (\mathcal{F})$.
\end{enumerate}

It then follows from \eqref{eqn:defnofderivedcouple} that:
\begin{eqnarray}
\nonumber
\bidegree(\partial^{(r+1)}) &=& \bidegree(\partial^{(r)}), \\
\label{eqn:bidegree-r}
\bidegree(i^{(r+1)})  &=& \bidegree(i^{(r)}), \\
\nonumber
\bidegree(j^{(r+1)}) &=& \bidegree(j^{(r)}) - \bidegree(i^{(r)}).				
\end{eqnarray}

Finally, it follows by induction from \eqref{eqn:bidegree-1} and \eqref{eqn:bidegree-r} that:

\begin{eqnarray*}
  \mathrm{bidegree} (\partial^{(r)}) & = &
%%  (- 1, - r),
(-1,0),
\\
  \mathrm{bidegree} (i^{(r)}) & = & (1, -1),\\
  \mathrm{bidegree} (j^{(r)}) & = & 
  %%(0, 0) .
  (-r+1,r-1).
\end{eqnarray*}
The  domains and codomains of $\partial^{(r)}_{p,q}, i^{(r)}_{p,q},j^{(r)}_{p,q}$ are as follows:
\begin{eqnarray}
\partial^{(r)}_{p,q} : E^{(r)}_{p,q} (\mathcal{F}) &\longrightarrow&
%%D^{(r)}_{n - 1, s - r} (\mathcal{F}), 
D^{(r)}_{p-1, q} (\mathcal{F}),  \nonumber  \\
%%sb changes
 i^{(r)}_{p,q} : D^{(r)}_{p,q} (\mathcal{F}) &\longrightarrow&
  D^{(r)}_{p+1, q-1} (\mathcal{F}), \label{eqn:Dr} \\
j^{(r)}_{p,q} : D^{(r)}_{p,q} (\mathcal{F}) &\longrightarrow& 
%%sb changes
%%E^{(r)}_{n,s} (\mathcal{F}) 
E^{(r)}_{p-r+1,q+r-1} (\mathcal{F}) 
\nonumber.
\end{eqnarray}

It then follows that
\begin{eqnarray*}
  d^{(r)}_{p,q} : E^{(r)}_{p,q} (\mathcal{F}) \longrightarrow E^{(r)}_{p-r, q+r-1} (\mathcal{F}), &  & 
\end{eqnarray*}
and thus,
\begin{eqnarray*}
  \mathrm{bidegree} (d^{(r)}) & = & (- r,r-1) .
\end{eqnarray*}
\begin{theorem}
  \label{thm:spectral-convergence}The spectral sequence defined above
  converges after a finite number of terms with
  \begin{eqnarray*}
    E_{p,q}^{\infty} (\mathcal{F}) & \cong &
    \ensuremath{\operatorname{Gr}}_{p} (H_{p+q} (X)) .
  \end{eqnarray*}
\end{theorem}

\begin{proof}
Since clearly the groups $E^{(1)}_{p,q} (\mathcal{F})
%%sb changes
= 0$  for all but finite number of pairs $(p,q)$, it is clear that
$d^{(r)}_{p,q} = 0$ for $r$ large enough, and this implies that the spectral
sequence $(E^{(r)} (\mathcal{F}), d^{(r)}
(\mathcal{F}))_{r \geq 1}$ converges unconditionally. The remaining part is a 
standard result (see for example \cite[Chapter 2]{Mcleary}).
\end{proof}

\begin{remark}
  \label{rem:infty}Notice that since in the filtration $\mathcal{F}$, $X_i =
  \emptyset$ for $i \leq 0$, we have that $d^{(r)}_{p,q} = 0$ for all $r \geq
  p$.
  %%sb changes
  Also, it follows from (\ref{eqn:E1}) that $E^{(1)}_{p,q} (\mathcal{F}) = 0$
  for all $p > N$. This implies that $E^{(r)}_{p,q} (\mathcal{F}) = 0$
  for all $p > N$ and $r \geq 1$. Thus,  $d^{(r)}_{p,n-p}$ and 
  %%$d^{(r)}_{p-r, n-p+r}$ 
  $d^{(r)}_{p+r, n-p-r+1}$ 
  are both $0$ for $r \geq \max (p, N - p + 1)$, and this implies that
  $E^{\infty}_{p,n-p} (\mathcal{F}) \cong E^{(\max (p, N - p + 1))}_{p, n-p}
  (\mathcal{F})$.
\end{remark}

The crucial link between the exact couples $\mathcal{C}^{(r)} (\mathcal{F})$
and the persistent homology groups is captured in the following observation.

\begin{lemma}
  \label{lem:basic}
  For 
  %%$r > 0, n, s \geq 0$
  $r,p,q \in \Z, r\geq 1$ and setting $n=p+q$,
  \begin{eqnarray*}
    %%D^{(r)}_{n, s} (\mathcal{F}) 
    D^{(r)}_{p-1, q+1} (\mathcal{F}) 
    & = & \ensuremath{\operatorname{Im}} (i_n^{p-r,
    p  - 1}) = H_n^{p-r, p  - 1} (\mathcal{F}),
  \end{eqnarray*}
  and
  \begin{eqnarray*}
 %%   i^{(r)}_{n, s} & = & i_n^{s + r - 1, s + r} \mid_{D^{(r)}_{n, s}(\mathcal{F})} .
  i^{(r)}_{p-1,q+1} & = & i_n^{p  - 1, p } \mid_{D^{(r)}_{p-1,q +1}(\mathcal{F})}.
  \end{eqnarray*}
\end{lemma}

\begin{proof}
For $r = 1$, both claims follow directly from the
definition of $D^{(1)}_{p,q} (\mathcal{F})$ (see Eqn. \eqref{eqn:D1} above) and
the definition of the derived couple (see Eqn. \eqref{eqn:defnofderivedcouple}). For $r > 1$,
the claim follows by induction. 
First notice that 
$
%%D^{(r)}_{n, s} = i^{(r -1)}_{n, s - 1} (D^{(r - 1)}_{n, s - 1})
D^{(r)}_{p-1,q+1}
  = i^{(r -1)}_{p - 2,q+2} (D^{(r - 1)}_{p -2,q+2})$ (Eqn. \eqref{eqn:Dr}). By induction
hypothesis we have
\begin{eqnarray*}
     %%D^{(r - 1)}_{n, s} (\mathcal{F})
     D^{(r - 1)}_{p-2,q+2} (\mathcal{F})
       & = & \tmop{Im} (i_n^{p-r, p  - 2}) =
     H_n^{p-r, p  - 2} (\mathcal{F}),
   \end{eqnarray*} 
and
\begin{eqnarray*}
  i^{(r - 1)}_{p - 2,q+2} & = & i^{p  - 2, p  - 1}_n \mid_{\tmop{Im}
  (i^{p-r, p  - 2}_n)}.
\end{eqnarray*}

It then follows from the definition of the derived couple (Definition
\ref{def:derived}) that
\begin{eqnarray*}
%%  D^{(r)}_{n, s} (\mathcal{F}) 
D^{(r)}_{p-1,q+1} (\mathcal{F})
& = & \tmop{Im} (i_n^{p-r, p - 1}) = H_n^{p-r,
  p - 1} (\mathcal{F}) \nocomma,
\end{eqnarray*}
and
\begin{eqnarray*}
  %%   i^{(r)}_{n, s} & = & i_n^{s + r - 1, s + r} \mid_{D^{(r)}_{n, s}(\mathcal{F})} .
  i^{(r)}_{p-1,q+1} & = & i_n^{p  - 1, p } \mid_{D^{(r)}_{p-1,q+1}(\mathcal{F})}.
 \end{eqnarray*}
   \end{proof}

\begin{theorem}
  \label{thm:basic} For each $r \geq 1, n=p+q$, the dimensions of the groups $E^{(\ast)}_{\ast, \ast}
  (\mathcal{F})$ and the persistent homology groups $H_\ast^{\ast, \ast} (\mathcal{F})$
  are related by the following long exact sequence.

\[ \cdots \rightarrow H^{p, p + r - 1}_n (\mathcal{F})
     \xrightarrow{j^{(r)}_{p+r-1,q-r+1}} E^{(r)}_{p,q}
     (\mathcal{F}) \xrightarrow{\partial^{(r)}_{p,q}}
     H_{n - 1}^{p - r, p - 1} (\mathcal{F})
     \xrightarrow{i^{(r)}_{p - 1,q}} H_{n - 1}^{p
     - r + 1, p} (\mathcal{F}) \rightarrow \cdots \]

  Moreover, for each $r \geq 1, n=p+q$, 
  %%$\ensuremath{\operatorname{Im}}(i^{(r)_{}}_{n, s}) = H^{s, s + r}_n (\mathcal{F})$.
 \[
 \ensuremath{\operatorname{Im}}(i^{(r)}_{p+r-1,q-r+1}) = H^{p, p + r}_n (\mathcal{F}).
 \]
\end{theorem}

\begin{proof}
Unravel the exact couple $\mathcal{C}^{(r)}
(\mathcal{F})$ and use Lemma \ref{lem:basic}.
\end{proof}

\section{Relations between $\dim_k E^r_{p,q} (\mathcal{F})$ and $b_n^{s, t}
(\mathcal{F})$}\label{sec:proofs}

In this section we prove the main theorems.

\begin{proof}[Proof of Theorem \ref{thm:spectral-to-persistent}]
Using
Theorem \ref{thm:basic}, Lemma \ref{lem:exact} and Lemma \ref{lem:basic} we
obtain that for 
%%each $r, n, s$ with $r > 0, n, s \geq 0$,
$r,p,q\in \Z, r\geq 1$ and $n=p+q$,

\begin{eqnarray*}
  \dim_k E^{(r)}_{p,q} (\mathcal{F}) & = &  (\dim_k D^{(r)}_{p+r-1,q-r+1} - \dim_k
  \tmop{Im} (i^{(r)}_{p+r-2,q-r+2})) +\\
  &  &  (\dim_k D^{(r)}_{p - 1,q} - \dim_k \tmop{Im}
  (i^{(r)}_{p- 1,q}))\\
  & = & (b_n^{p, p + r - 1} (\mathcal{F}) - b_n^{p - 1, p + r - 1}
  (\mathcal{F})) + (b_{n - 1}^{p - r, p - 1} (\mathcal{F}) - b_{n - 1}^{p - r,
  p} (\mathcal{F})).
\end{eqnarray*}
\end{proof}

\begin{remark}
  Notice that for $r \geq \max (p, N - p + 1)$ we have
  \begin{eqnarray*}
    b_n^{p, p + r - 1} (\mathcal{F}) & = & b_n^{p, N} (\mathcal{F}),\\
    b_n^{p - 1, p + r - 1} (\mathcal{F}) & = & b_n^{p - 1, N} (\mathcal{F}),\\
    b_{n - 1}^{p - r, p - 1} (\mathcal{F}) & = & 0,\\
    b_{n - 1}^{p - r, p} (\mathcal{F}) & = & 0.
  \end{eqnarray*}
  Using Remark \ref{rem:infty} we verify that
  \begin{eqnarray*}
    \dim_k E^{(\infty)}_{p,q} (\mathcal{F}) & = & \dim_k E^{(\max (p, N - p +
    1))}_{p,q} (\mathcal{F})\\
    & = & b_n^{p, N} (\mathcal{F}) - b_n^{p - 1, N} (\mathcal{F})\\
    & = & \dim_k \ensuremath{\operatorname{Im}} (i_n^{p, N}) - \dim_k
    \ensuremath{\operatorname{Im}} (i_n^{p - 1, N})\\
    & = & \dim_k  (\ensuremath{\operatorname{Im}} (i_n^{p, N})
    /\ensuremath{\operatorname{Im}} (i_n^{p - 1, N}))\\
    & = & \dim_k \ensuremath{\operatorname{Gr}}_{p} (H_n (X)) .
  \end{eqnarray*}
\end{remark}

\begin{proof}[Proof of Corollary \ref{cor:main}]
Denote for $n \geq 0$, and $s, t \in \mathbb{Z}$
\begin{eqnarray*}
  \gamma_n^{s, t} (\mathcal{F}) & = & (b_n^{s - t, s - 1} (\mathcal{F}) -
  b_n^{s - t, s} (\mathcal{F})),\\
  \nu_n^{s, t} (\mathcal{F}) & = & (b^{s, s + t - 1}_n (\mathcal{F}) - b^{s -
  1, s + t - 1}_n (\mathcal{F})) .
\end{eqnarray*}
It follows from 
%%Eqn. (\ref{eqn:main}) 
Theorem \ref{thm:spectral-to-persistent} 
that
\[ \sum_s \dim_k E^{(r)}_{s,n-s} (\mathcal{F}) = \sum_{0 \leq s \leq N + 1}
   (\nu_n^{s, r} (\mathcal{F}) + \gamma_{n - 1}^{s, r} (\mathcal{F})) . \]
Also, summing both sides of Eqn. \eqref{eqn:mu} we obtain two different
expressions for $\sum_{j - i \geq r} \mu^{i, j}_n (\mathcal{F})$, one in terms
of the $\nu_n^{s, t}$ , and the other in terms of the $\gamma_n^{s, t} .$

More precisely:
\begin{eqnarray}
  \sum_{j - i \geq r} \mu^{i, j}_n (\mathcal{F}) & = & \sum_{j - i \geq r}
  ((b^{i, j - 1}_n (\mathcal{F}) - b^{i, j}_n (\mathcal{F})) - (b^{i - 1, j -
  1}_n (\mathcal{F}) - b^{i - 1, j}_n (\mathcal{F}))) \nonumber\\
  & = & \sum_{r \leq t \leq N + 1} \sum_{0 \leq s \leq N + 1} (\gamma_n^{s,
  t} (\mathcal{F}) - \gamma_n^{s, t + 1} (\mathcal{F})),  \label{eqn:gamma}
\end{eqnarray}

\begin{eqnarray}
  \sum_{j - i \geq r} \mu^{i, j}_n (\mathcal{F}) & = & \sum_{j - i \geq r}
  ((b^{i, j - 1}_n (\mathcal{F}) - b^{i - 1, j - 1}_n (\mathcal{F})) - (b^{i,
  j}_n (\mathcal{F}) - b^{i - 1, j}_n (\mathcal{F}))) \nonumber\\
  & = & \sum_{r \leq t \leq N + 1} \sum_{0 \leq s \leq N + 1} (\nu_n^{s, t}
  (\mathcal{F}) - \nu_n^{s, t + 1} (\mathcal{F})) .  \label{eqn:nu}
\end{eqnarray}

It follows first by changing the order of summation, and then using
telescoping on Eqn. \eqref{eqn:gamma} that
\begin{eqnarray}
  \sum_{j - i \geq r} \mu^{i, j}_n (\mathcal{F}) & = & \sum_{0 \leq s \leq N +
  1} \sum_{r \leq t \leq N + 1} (\gamma_n^{s, t} (\mathcal{F}) - \gamma_n^{s,
  t + 1} (\mathcal{F}))  \label{eqn:cormain2} \nonumber\\
  & = & \sum_{0 \leq s \leq N + 1} (\gamma_n^{s, r} (\mathcal{F}) -
  \gamma_n^{s, N + 1}) \nonumber\\
  & = & \sum_{0 \leq s \leq N + 1} \gamma_n^{s, r} (\mathcal{F}) - \sum_{0
  \leq s \leq N + 1} ((b_n^{s - N - 1, s - 1} (\mathcal{F}) - b_n^{s - N - 1,
  s} (\mathcal{F}))) \nonumber\\
  & = & \sum_s \gamma_n^{s, r} (\mathcal{F}) . %%\nonumber
\end{eqnarray}
Similarly from Eqn. \eqref{eqn:nu}  we get that
\begin{eqnarray}
  \sum_{j - i \geq r} \mu^{i, j}_n (\mathcal{F}) & = & \sum_{0 \leq s \leq N +
  1} \sum_{r \leq t \leq N + 1} (\nu_n^{s, t} (\mathcal{F}) - \nu_n^{s, t + 1}
  (\mathcal{F}))  \label{eqn:cormain3} \nonumber\\
  & = & \sum_{0 \leq s \leq N + 1} (\nu_n^{s, r} (\mathcal{F}) - \nu_n^{s, N
  + 1}) \nonumber\\
  & = & \sum_{0 \leq s \leq N + 1} \nu_n^{s, r} (\mathcal{F}) - \sum_{0 \leq
  s \leq N + 1} (b^{s, s + N}_n (\mathcal{F}) - b^{s - 1, s + N}_n
  (\mathcal{F})) \nonumber\\
  & = & \sum_s \nu_n^{s, r} (\mathcal{F}) - b_n (X) . %%\nonumber
\end{eqnarray}
The corollary follows from Eqns. \eqref{eqn:cormain2} and
(\ref{eqn:cormain3}).
\end{proof}

\begin{proof}[Proof of Theorem \ref{thm:persistence-in-terms-of-spectral}] 
It follows from Theorem
\ref{thm:spectral-convergence} and (\ref{eqn:graded}) that for each $n \geq 0$
and $i \leq t$,
\begin{eqnarray*}
  \dim_k E^{\infty}_{i,n-i} (\mathcal{F}_{\leq t}) & = & b_n^{i, t}
  (\mathcal{F}) - b_n^{i - 1, t} (\mathcal{F}) .
\end{eqnarray*}
The theorem follows after taking the sum of both sides over all $i, 0 \leq i
\leq s$, and noting that $b_n^{i, t} (\mathcal{F}) = 0$ for $i \leq 0$.
Finally, by Remark \ref{rem:infty} we have that $E^{\infty}_{i,n-i}
(\mathcal{F}_{\leq t}) \cong E^{(\max (i, t - i + 1))}_{ i,n-i}
(\mathcal{F}_{\leq t})$.
\end{proof}

\section*{Acknowledgments}{The authors thank Estanislao Herscovich, Ana Romero and an anonymous referee for several comments and corrections.}

\bibliographystyle{amsplain}
\bibliography{master}

\end{document}